\documentclass[11pt]{article}
\usepackage{fullpage}
\usepackage{fancyhdr}
\usepackage{amsmath,amssymb,amsthm}
\usepackage{amsmath,amscd}
\usepackage{mathtools}
\usepackage{mathrsfs}
\usepackage{tikz-cd}
\usepackage{authblk}
\usepackage{enumitem}

\usepackage{stmaryrd}

\usepackage[letterpaper, margin=1in]{geometry}
\usepackage{mathpazo}
\usepackage{euler}

\usepackage{hyperref}

\newcommand{\Z}{\mathbb{Z}}

\newcommand{\C}{\mathbb{C}}

\newcommand{\Q}{\mathbb{Q}}

\newcommand{\Spec}{\text{Spec}}

\newcommand{\brac}[1]{\big[ #1 \big] }
\newcommand{\dbrac}[1]{\llbracket #1 \rrbracket}
\newcommand{\cbrac}[1]{\{ #1 \}}

\newcommand{\Nil}{\textup{Nil}}

\newcommand{\dia}{\textbf{dg}}
\newcommand{\had}{\textbf{hg}}

\newtheorem{theorem}{Theorem}[section]
\newtheorem{lemma}[theorem]{Lemma}

\newtheorem{proposition}[theorem]{Proposition}

\newtheorem{corollary}[theorem]{Corollary}

\newtheorem{question}[theorem]{Question}
\newtheorem{example}[theorem]{Example}

\theoremstyle{definition}

\newtheorem{remark}[theorem]{Remark}
\newtheorem{definition}[theorem]{Definition}

\def\frak{\relaxnext@\ifmmode\let\next\frak@\else
 \def\next{\Err@{Use \string\frak\space only in math mode}}\fi\next}
\def\goth{\relaxnext@\ifmmode\let\next\frak@\else
 \def\next{\Err@{Use \string\goth\space only in math mode}}\fi\next}
\def\frak@#1{{\frak@@{#1}}}
\def\frak@@#1{\noaccents@\fam\euffam#1}
\font\tengoth=eufm10
\newfam\gothfam \def\goth{\fam\gothfam\tengoth} \textfont\gothfam=\tengoth

\title{The diagonal and Hadamard grade of hypergeometric functions}
\author{Andrew Harder and Joe Kramer-Miller}

\date{}

\begin{document}
\maketitle
\abstract{Diagonals of rational functions are an important class of functions arising in number theory, algebraic geometry, combinatorics, and physics. In this
paper we study the diagonal grade of a function $f$, which is defined to be the smallest
$n$ such that $f$ is the diagonal of a rational function in variables $x_0,\dots, x_n$. We relate the diagonal grade of a function
to the nilpotence of the associated differential equation. This allows
us to determine the diagonal grade of many hypergeometric functions
and answer affirmatively the outstanding question on the existence of
functions with diagonal grade greater than $2$.
In particular, we show that $\prescript{}{n}F_{n-1}(\frac{1}{2},\dots, \frac{1}{2};1\dots,1 \mid x)$ has diagonal grade $n$ for each $n\geq 1$.
Our method also applies to the generating function of the Apéry sequence, 
which we find to have diagonal grade $3$. We also answer related questions
on Hadamard grades posed by Allouche and Mendès France. 
For example, we show that $\prescript{}{n}F_{n-1}(\frac{1}{2},\dots, 
\frac{1}{2};1\dots,1 \mid x)$ has Hadamard grade $n$ for all $n\geq 1$.
}
\section{Introduction}
Fix a finite extension $K$ of $\Q$ and consider $h=\frac{P}{Q}$, where $P,Q \in K[x_0,\dots,x_n]$. If $Q(0,\dots,0) \neq 0$, there is a Taylor series expansion of $h$ at $(0,\dots,0)$ given by
\begin{equation*}
    h = \sum_{i_0,\dots,i_n\geq 0} a_{i_0,\dots,i_n}x_0^{i_0}\dots x_n^{i_n}.
\end{equation*}
The diagonal of $h$ is the univariate power series defined by
\begin{align*}
    \Delta_n(h) = \sum_{i\geq 0} a_{i,\dots,i} x^i.
\end{align*}
Diagonals are an important class of functions arising in number theory, algebraic geometry, combinatorics, and physics. For instance, many hypergeometric functions are diagonals as well as the generating function for Apéry's sequence. An important conjecture of Christol
states that $f$ is a diagonal if and only if $f$ is globally bounded (i.e. $f$ is $D$-finite, has a positive radius of convergnece in $\C$, and there exists $N,c \geq 0$ with $cf(Nx) \in \Z\dbrac{x}$). It is natural to ask about the minimum number
of variables needed to express a function as a diagonal. This motivates the following definition:
\begin{definition}
    Let $f(x) \in K\dbrac{x}$ be a nonzero power series. The diagonal grade of $f(x)$, written
    $\dia(f)$, is the smallest number $n$ such that there is a rational
    function $h$ in variables $x_0,\dots,x_n$ with $f=\Delta_n(h)$.
    If no such $n$ exists we take $\dia(f)=\infty$.
\end{definition}
We also introduce the following $K$-vector spaces:\footnote{Some sources, such as \cite{melczer2021invitation}, define $\mathcal{D}_k$ to be
diagonals of $k$ variables instead of $k+1$ variables. Our choice of indexing is motivated by \S 4, which makes $\mathcal{D}$ into a filtered ring.}
\begin{align*}
    \mathcal{D}_k &\coloneqq \left \{ f(x) \in K\dbrac{x} ~\middle | ~\dia(f) \leq k \right \}, \\
    \mathcal{D} &\coloneqq \left \{ f(x) \in K\dbrac{x} ~\middle |~ \dia(f) <\infty \right \}.
\end{align*}
    Note that $\mathcal{D}_0$ is the space of rational functions
    and by Furstenberg's theorem \cite{Furstenberg} we know $\mathcal{D}_1$ is the
    space of algebraic functions. It is also known that $\mathcal{D}_1 \subsetneq \mathcal{D}_2$, as the hypergeometric function $\prescript{}{2}F_{1}(\frac{1}{2}, \frac{1}{2};1 \mid x)$ can be expressed as a diagonal in three variables, but is known to be transcendental. Thus, we may think of the diagonal grade
    as a trascendental measure for globally bounded series. 
    However, in general it is an open question $\mathcal{D}_k \subsetneq \mathcal{D}_{k+1}$ for all $k$ or even if $\mathcal{D}_2 \subsetneq \mathcal{D}_k$ for some $k\geq 3$ (see e.g. \cite[\S 3]{Bostan-Lairez-Salvy-Multiple_binomial_sums} or \cite[Open Problem 3.1]{melczer2021invitation}.) In this article, we give an affirmative answer to this question. More precisely, we establish
    the diagonal grade of many well-known diagonals: E.g. we show that the hypergeometric
    functions $\prescript{}{n}F_{n-1}(\frac{1}{2},\dots, \frac{1}{2};1\dots,1 \mid x)$ has diagonal grade $n$ (see Theorem \ref{t:all beta = 1}) and we show that
    the generating function for Ap\'ery's sequence has diagonal grade $3$ (see Example \ref{e: Apery}). 

    A closely related notion of Hadamard grade was introduced
    by Allouche and Mendès France \cite{allouche2011hadamard}.
    Recall that the Hadamard product of two series $f(x) = \sum a_nx^n$
    and $g(x) = \sum b_nx^n$ is defined as 
    \begin{align*}
        f*h &= \sum a_nb_n x^n.
    \end{align*}
    In general, the Hadamard product of two diagonals is again a
    diagonal. 
    \begin{definition}
        Let $f(x) \in K\dbrac{x}$. The Hadamard grade of $f(x)$,
        which we denote by $\had(f)$, is defined as $0$ if $f$ is rational and is otherwise defined as the smallest $k$ such that there exists
        algebraic series $h_1(x),\dots,h_k(x)$ with
        \begin{align*}
            f(x) &= h_1* \dots *h_k.
        \end{align*}
        If no such $k$ exists we set $\had(f)=\infty$.
        \end{definition}
    As before, we define the following sets:
    \begin{align*}
    \mathcal{H}_k &\coloneqq \left \{ f(x) \in K\dbrac{x} ~\middle |~ \had(f) \leq k \right \}, \\
    \mathcal{H} &\coloneqq \left \{ f(x) \in K\dbrac{x} ~\middle |~ \had(f) <\infty \right \}.
\end{align*}
    However, we emphasize that it is not at all clear if $\mathcal{H}_k$
    or $\mathcal{H}$ are vector spaces. Allouche and Mendès France
    ask similar questions to those asked above about diagonal grades. For example, is $\mathcal{H}_k \subsetneq \mathcal{H}_{k+1}$ for all $k$? Equivalently, they ask if there exists a series $f$ with $\had(f)=k$ for every $k\geq 0$
    and specifically ask if $\prescript{}{n}F_{n-1}(\frac{1}{2},\dots, \frac{1}{2};1\dots,1 \mid x)$ has Hadamard grade $n$. 
    Assuming the very difficult Lang--Rohrlich conjecture
    on algebraic independence of gamma values, Rivoal and Roques
    established the Hadamard grade of many hypergeometric functions in
    \cite{rivoal2014hadamard}. In particular, they affirmatively answer
    both questions stated above. In this article, we unconditionally show $\mathcal{H}_k \subsetneq \mathcal{H}_{k+1}$ and determine the Hadamard grade of several hypergeometric functions. 

    This article is organized as follows. In \S \ref{s: local monodromy bounds} we establish
    lower bounds on $\dia(f)$ and $\had(f)$ in terms of
    the local monodromy of the differential equation satisfied by $f$.
    This lower bound, in particular, establishes `half' of a recent conjecture 
    of put forth by a group of physicists \cite{hassani2025diagonals_conjecture}. The key ingredients to this lower
    bound are Deligne's open monodromy theorem and the realization
    of diagonals as solutions to certain Picard--Fuchs equations.
    In \S \ref{s: applications} we use the lower bound from
    \S \ref{s: local monodromy bounds} to bound the
    diagonal grade and Hadamard grade of a large class
    of hypergeometric functions. In many interesting cases, we
    can completely determine the diagonal grade and Hadamard grade,
    thus showing the strict inclusions $\mathcal{D}_k \subsetneq \mathcal{D}_{k+1}$ and $\mathcal{H}_k \subsetneq \mathcal{H}_{k+1}$. As a proof of
    concept, we also show that
    the generating function for Apéry's sequence has diagonal grade $3$
    and give some explicit examples arising in physics.
    The same reasoning works for many of the examples put forth in \cite{hassani2025diagonals_conjecture} and
    \cite{tables_of_calabi_yau}.
    Finally, in \S \ref{s: the ring of diagonals} we investigate
    some basic properties of the ring $\mathcal{D}$, where multiplication
    is given by Hadamard product. We show that $\mathcal{D}_{k_1}*\mathcal{D}_{k_2} \subset \mathcal{D}_{k_1+k_2}$. We also show that the zero divisors
    of $\mathcal{D}$ can be completely explained by the action of
    roots of unity on $\mathcal{D}$.

\section{Local monodromy bounds} \label{s: local monodromy bounds}
    For this section we fix a finite extension $K$ of $\Q$.

    \subsection{Background on differential modules and nilpotent monodromy}
    \begin{definition}
        A differential module over $K(x)$ is a finite dimensional
        vector space $M$ and a $K$-linear map $\partial:M \to M$
        satisfying the Leibnitz rule: 
        \[\partial(fm) = x\frac{d}{dx}f \cdot m + f \cdot \partial(m).\]
        We point the reader to \cite[Chapter 2]{vanderPut-Singer-book}
        for the basic theory of differential modules.
    \end{definition}
	\begin{definition}
	    Let $M$ be a differential module over $K(x)$. We say that $M$
        is \emph{regular} at $x=0$ if there is a $k[x]$-lattice $M_0$ (i.e.
        a $K[x]$-module $M_0 \subset M$ with $M_0 \otimes_{K[x]} K(X) = M$)
        such that $\partial(M_0) \subset M_0$. By formal
        Fuchsian theory (see \cite{Dwork-G-functions}), the $K$-linear
        map on $T:M_0/xM_0 \to M_0/xM_0$ decomposes $T=UD=DU$, where $D$
        is a diagonal matrix and $U$ is unipotent. We define the
        \emph{nilpotence index}, denoted by $\Nil(M)$, to be the smallest 
        integer $k$ such that $(U-1)^k=0$. We remark
        that $\Nil(M)=\Nil(M^\vee)$, where $M^\vee$ is the dual
        differential module.
	\end{definition}
    \begin{definition}\label{d: nilpotence exponent}
        Let $f$ be a $D$-finite element of $k\dbrac{x}$ (i.e.
        a series annihilated by an element of $k(x)[\partial]$) and
        let $L_f = \partial^k + a_{k-1}\partial^{k-1} + \dots + a_0$
        be the minimal differential operator annihilating $f$. 
        Let $M_f$ be the differential module associated to $L_f$. We define
        the exponent of nilpotence of $f$ to be $\Nil(f)=\Nil(M_f)$
    \end{definition}

    \begin{lemma}\label{l: bounding nilpotence of individual entries}
		Let $M$ be a differential module over $K(x)$. Let
		$s:M \to K\dbrac{x}$ be a horizontal map. Then
		for any $\alpha \in s(M)$ we have $\Nil(\alpha) \leq \Nil(M)$.
	\end{lemma}
	\begin{proof}
		Let $e_0 \in s^{-1}(\alpha)$ and let $N$ be the smallest
		differential module over $K(x)$ containing $e_0$. Let
		$r$ be the rank of $N$ and let $e_i = D^i(e_0)$. Then
		$e_0, \dots,e_{r-1}$ form a basis of $N$ over $K(x)$
		and $s(e_i) = \frac{d}{dx}^i \alpha$.
		In particular, we see that $\sum_{i=0}^{s-1} \alpha^{(i)}e_i^\vee$
		is a solution to $N^\vee$, where $e_0^\vee, \dots, e_{s-1}^\vee$
		denotes the basis dual to $e_0,\dots,e_{s-1}$. Thus,
		$N^\vee$ is the differential module associated
		to the minimal differential operator satisfied by $\alpha$. Thus,
		\begin{equation*}
			\Nil(\alpha) = \Nil(N^\vee) = \Nil(N) \leq \Nil(M).\qedhere
		\end{equation*}
	\end{proof}

    \subsection{Main results}
    We define $K\cbrac{x_0,\dots,x_n}$ to be the subring of $K\dbrac{x_0,\dots,x_n}$
    consisting of series that are algebraic over $K(x_0,\dots,x_n)$.
    We freely use the theory of modules with connection over an algebraic
    variety (see e.g. \cite[S 1]{katz_nilpotent_monodromy_theorem}
    for a quick introduction.)
    \begin{theorem}
        \label{t: nilpotence bound for diagonal}
        Let $h \in K\cbrac{x_0,\dots,x_n}$ 
        and set
        \[f(x) = \Delta_n(h(x)). \]
        Then $\Nil(f) \leq n+1$ and if $h$ is rational then $\Nil(f) \leq n$.
    \end{theorem}
    \begin{proof}
        Since $h$ is algebraic, there exists an
        affine open subspace $U$ of $\Spec(K[x_0^{\pm 1},\dots,x_n^{\pm 1}])$
        and a module with integrable connection $(N,\nabla)$ on $U$ such that:
        \begin{enumerate}
            \item $N$ is irreducible.
            \item There exists $e \in N$ such that the map
            $(N,\nabla) \to K\dbrac{x_0,\dots,x_n}$ sending $e \mapsto h$ is 
            horizontal, where $K\dbrac{x_0,\dots,x_n}$ has the trivial
            connection.
        \end{enumerate}
        After shrinking $U$, there exists a finite \'etale map $\psi: V \to U$
        so that the function field of $V$ contains $h$. Note that
        $N$ is a subobject of $\psi_*(\mathcal{O}_V,d)$, the pushforward
        of the trivial module with connect on $V$ along the map $\psi$.
        Let $\widetilde{U}$ (resp. $\widetilde{V}$) denote the fiber
        of $U$ (resp. $V$) and $\Spec(K[x][x_0,\dots,x_n]/x_0\dots x_n-x)$ over
       $\Spec(K[x][x_0,\dots,x_n])$. Let $\widetilde{\psi}: \widetilde{V} \to 
       \widetilde{U}$ denote the pullback of $\psi$ and note that
       $\widetilde{\psi}$ is a finite \'etale map. We set $\widetilde{N}$ to be the pullback
       of $N$ to $\widetilde{U}$ and again we note that $\widetilde{N}$
       is a subobject of $\widetilde{\psi}_*(\mathcal{O}_{\widetilde{V}},d)$.

        Let $\pi:\widetilde{U} \to \Spec(K[x])$ be the natural map. The Leray spectral sequence gives $R^{n}(\pi \circ \widetilde{\psi})_*(\mathcal{O}_{\widetilde{V}},d)
	=R^{n} \pi_* (\widetilde{\psi}_*(\mathcal{O}_{\widetilde{V}},d))$, so we
	see that $R^{n}\pi_* \widetilde{N}$ is a subojbect of 
	$R^{n}\psi_*(\mathcal{O}_{\widetilde{V}},d)$. By Hironaka's
    resolution of singularities, after replacing $K(X)$
    with a finite extension $L$, there is a proper and smooth
    morphism $\rho: X \to \Spec(L)$ and a divisor $Y$ with normal crossings
    such that $X-Y \cong \widetilde{V} \times L$ and the restriction
    of $\rho$ to $X-Y$ is the pullback of $\pi \circ \widetilde{\psi}$
    along $\Spec(L) \to \Spec(k[x])$.
    Let 
    $h(n)$ denote the number of $q$ such that $\dim_{L}H^{q}(X, \Omega^{n-q}_{X/L}(\log Y))$ is nonzero. 
    By Deligne's open
	local monodromy theorem \cite[Theorem 14.3]{katz_nilpotent_monodromy_theorem}, 
	$R^{n}\psi_*(\mathcal{O}_{\widetilde{V}},d)$ with the
	Gauss-Manin connection has nilpotence index bounded by $h(n)$. Thus, 
    $R^{n}\pi_* \widetilde{N}$ also has nilpotence index bounded above
    by $h(n)$. 

    By work of Christol (see \cite[Section 5]{christol1990bounded} or \cite[I.3.3]{andre1898gfunctions}) there is a
	horizontal map $\theta: R^{n}\pi_*\widetilde{N} \to k\dbrac{x}$ that
	sends the cohomology class
	\begin{equation*}
	e \otimes \frac{dx_1 \dots d{x_n}}{x_1\dots x_n} \mapsto f=\Delta(h). 
	\end{equation*}
    Then from Lemma \ref{l: bounding nilpotence of individual entries}
    and the previous paragraph, we see that $\Nil(f) \leq h(n)$. 
    In the general case, we see that $h(n)\leq n+1$, since any valid
    $q$ must satisfy $0 \leq q \leq n$. In the case that $h$ is rational,
    we can take $V=U$ and $\widetilde{V}=\widetilde{U}$. In particular,
    the function field of $\widetilde{V}$ is $L(x_1,\dots,x_{n})$,
    so that $X$ is rational. Thus, $H^{n}(X, \mathcal{O}_X)$ vanishes,
    which implies $h(n)\leq n$.\qedhere
       
    \end{proof}

    \begin{corollary}
        \label{c: bound on diagonal grade}
        Let $f$ be a $D$-finite series in $K\dbrac{x}$. Then $\dia(f) \geq \Nil(f)$
        and $\had(f) \geq \Nil(f)$.
    \end{corollary}
    \begin{proof}
        The statement for $\dia(f)$ is immediate from Theorem \ref{t: nilpotence bound for diagonal}. The statement for $\had(f)$ follows from
        Theorem \ref{t: nilpotence bound for diagonal} using
        the following observation: for $f_0,\dots,f_{n-1} \in K\cbrac{x}$ 
        we have 
        \begin{align*}
            f_0(x)* \dots * f_{n-1}(x) &= \Delta_{n-1}(f_0(x_0)\cdot \dots \cdot f_{n-1}(x_{n-1})),
        \end{align*}
        so that the Hadamard product $f_0(x)* \dots * f_{n-1}(x)$
        is the diagonal of an element in $K\cbrac{x_0,\dots,x_{n-1}}$. 
    \end{proof}
    \begin{remark}
        We remark that Corollary \ref{c: bound on diagonal grade}
        proves `half' of the conjecture described in \cite[\S 1.2]{hassani2025diagonals_conjecture}. In particular,
        the highest power of $\log(x)$ occurring in the solution space
        of $L_f$ around $x=0$ equals the nilpotence index, which by
        Corollary \ref{c: bound on diagonal grade} gives a lower bound on
        the diagonal grade. The other direction of this conjecture is
        likely to be much more difficult. It is also worth pointing out
        the relation with \cite[Conjecture 1.3.6]{Lam-Litt-Algebraic_solutions_to_diffeqs}, when restricted to
        connections that admit at least one globally bounded solution.
        This conjecture states that finite local monodromy implies
        algebraicity.
    \end{remark}

\section{Applications} \label{s: applications}
We now apply Corollary \ref{c: bound on diagonal grade} to study 
the diagonal and Hadamard grade of some well known globally bounded functions.
\subsection{Hypergeometric functions}

\paragraph{Basic definitions} Let $\alpha = (\alpha_1,\dots, \alpha_n) \in \mathbb{Q}^n, \beta = (\beta_1,\dots, \beta_{n}) \in \mathbb{Q}^{n}$, with $\beta_n=1$ and $\alpha_i \neq \beta_j$ if $i \neq j$. Let $(x)_i = x(x+1)\dots (x+i-1)$ for any number $x$ and $i\in \mathbb{N}$. We have the usual hypergeometric functions of a single variable:
\[
\prescript{}{n}F_{n-1}\left(\alpha;\beta \mid x\right) = \sum_{i=0}^\infty\left( \dfrac{(\alpha_1)_i\dots (\alpha_n)_i}{(\beta_1)_i\dots (\beta_{n-1})_i (1)_i}\right) x^i.
\]
We observe directly that 
\[
\prescript{}{n}F_{n-1}\left(\alpha;\beta \mid x\right) = \prescript{}{2}F_1(\alpha_1,1;\beta_1 \mid x)* \dots  * \prescript{}{2}F_1(\alpha_n,1;\beta_n,1 \mid x).
\]
In the specialized case $\beta=(1,\dots,1)$ we have 
    \begin{equation}\label{eq: hadamard product when betas are one}
    \prescript{}{n}F_{n-1}(\alpha;1,\dots, 1\mid  x) = \prescript{}{1}F_0(\alpha_1;1 \mid x)* \dots  * \prescript{}{1}F_0(\alpha_n;1 \mid x).
    \end{equation}
     The functions $\prescript{}{1}F_0(\alpha_i; 1\mid x) = (1-x)^{-\alpha_i}$ are algebraic if $\alpha_i\in \mathbb{Q}$. Therefore, the Hadamard grade of $\prescript{}{n}F_{n-1}(\alpha;1,\dots, 1\mid x)$ is at most $n$. 

Let $\theta = x \frac{d}{dx}$. Then $\prescript{}{n}F_{n-1}\left(\alpha;\beta \mid x\right)$ is a solution to the differential operator
\[
L({\alpha}; {\beta}) = \theta\prod_{i=1}^{n-1}(\theta + \beta_i - 1) + x \prod_{i=1}^n (\theta +  \alpha_i).
\]
Furthermore, the singularities of $L(\alpha;\beta)$ occur at $x=0,1,\infty$ and are regular.
We define $V({ \alpha};{\beta})$ to be the space of solutions of $L(\alpha;\beta)$ at a
point $x_0 \neq 0,1,\infty$. For $*=0,1,\infty$ we let $T_*$ denote the monodromy action
on $V({\alpha};{\beta})$ by analytic continuation around a simple counterclockwise loop containing $*$.

\paragraph{The nonresonant case} When $\alpha_i - \beta_j \in \mathbb{Z}$ for a pair of $i,j$ we say that $({\alpha};{\beta})$ is resonant. In the nonresonant case, the monodromy representation of $V({ \alpha};{\beta})$ has been computed by Levelt \cite{levelt1961hypergeometric}, and is irreducible \cite[Proposition 3.3]{beukers1989monodromy}.

According to Levelt, if we let $a_i = \exp(2\pi {\tt i}\alpha_i), b_i = \exp(2\pi {\tt i} \beta_i)$ and $a(t) = \prod (t-a_i), b(t) = \prod (t-b_i)$, then $T_\infty = \mathcal{C}_{a(t)}$ and $T_0 = \mathcal{C}_{b(t)}$ where $\mathcal{C}_{f(t)}$ denotes the companion matrix (e.g. \cite[pp. 475]{dummit2004abstract}) of the polynomial $f(t)$. Consequently, $T_\infty$ has a single Jordan block corresponding to each distinct root of $a(t)$ whose dimension is the multiplicity of that root. Similarly, $T_0$ has a unique Jordan block for each distinct root of $b(t)$. Therefore we have the following statement.
\begin{proposition}\label{p: nonresonant nilpotence}
    Suppose $({\alpha},{\beta})$ is nonresonant. The hypergeometric differential equation $L({\alpha},{\beta})$ has nilpotence index $n$ if ${\beta} \in \mathbb{Z}^{n-1}$. Moreover, if $a_1= \dots = a_n = m/k$ with $m,k \in \mathbb{Z}$ then $T_\infty^k$ is maximally unipotent.
\end{proposition}

We need the following lemma comparing diagonal and Hadamard grades.
 \begin{lemma}
     \label{l: relate hg and dg}
     Let $f \in K\dbrac{x}$ be a $D$-finite series. The following
     are equivalent:
     \begin{enumerate}[label=(\Roman*)]
         \item There exists an 
         $g \in K\cbrac{x_0,\dots,x_n}$ with $f=\Delta_n(g)$.
         \item There is a rational function
         $h \in K\brac{x_0,\dots,x_{n+1}}_{(x_0,\dots,x_{n+1})}$ with $f = \Delta_{n+1}(h)$.
     \end{enumerate}
     In particular, we have
     $\dia(f) \leq \had(f)$.
 \end{lemma}

\begin{proof}
    Define a map $\mathscr{D}:K\dbrac{x_0,\dots,x_{n+1}} \to K\dbrac{x_0,\dots,x_{n}}$ by
    \begin{equation*}
        \mathscr{D}\left (\sum_{i_0,\dots,i_{n+1}\geq 0} a_{i_0}\dots _{i_{n+1}}x_0^{i_0}\dots x_{n+1}^{i_{n+1}} \right ) = \sum_{\substack{i_0,\dots,i_{n+1}\geq 0 \\ i_0 +\dots+ i_n = i_{n+1}}} a_{i_0,\dots,i_{n+1}} x_{0}^{i_0}\dots x_{n}^{i_n}.
    \end{equation*}
    Note that
    \begin{equation}
        \label{eq: D and diagonal}
        \Delta_n\left (\mathscr{D}\left (\sum_{i_0,\dots,i_{n+1}\geq 0} a_{i_0}\dots _{i_{n+1}}x_0^{i_0}\dots x_{n+1}^{i_{n+1}} \right )\right) = \sum_{i\geq 0} a_{i,\dots,i,(n+1) i} x^i
    \end{equation}
    By \cite[Theorem 6.2, Remark 6.4]{Denef-Lipshitz-alg_series_diagonals} we know that $\mathscr{D}$ maps $K[x_0,\dots,x_{n+1}]_{(x_0,\dots,x_{n+1})}$ surjectively onto
    $K\{x_0,\dots,x_n \}$.  
    First, assume  $(II)$ holds, so let $h$ is a rational function with $f = \Delta_{n+1}(h)$. Then
    from \eqref{eq: D and diagonal} we have \[\Delta_n(\mathscr{D}(h(x_0,\dots,x_n,x_{n+1}^{n+1}))= f.\]
    As $\mathscr{D}(h(x_0,\dots,x_n,x_{n+1}^{n+1})))$ is algebraic, this implies $(I)$. Next, assume $(I)$ holds, so that
    there is an algebraic function $g$ with $f = \Delta_n(g)$. 
    We know there exists $r \in K\brac{x_0,\dots,x_{n+1}}_{(x_0,
    \dots,x_{n+1})}$ with $\mathscr{D}(r)=g$. 
    Let $s \in K\brac{x_0,\dots,x_{n+1}}_{(x_0,\dots,x_{n+1})}$
    be the rational function satisfying
    \begin{align*}
        s(x_0,\dots,x_{n},x_{n+1}^{n+1}) &= \frac{1}{n+1}\sum_{\zeta^{n+1}=1} r(x_0,\dots,x_n,\zeta x_{n+1}).
    \end{align*}
    By our definition of $s$ and \eqref{eq: D and diagonal} we compute:
    \[ \Delta_{n+1}(s) = \Delta_n(\mathscr{D}(r))=\Delta_n(g)=f,\]
    which implies $(II)$.    
\end{proof}
 
\begin{theorem}\label{t:all beta = 1}
Suppose $({\alpha}, {\beta})$ is nonresonant and ${\beta} = (1,\dots, 1)$ or 
if $\alpha_1=\dots = \alpha_n$ then:
\[\dia(\prescript{}{n}F_{n-1}({\alpha}; {\beta}\mid x)) = \had(\prescript{}{n}F_{n-1}({\alpha}; {\beta}\mid x)) = n.\]
\end{theorem}

\begin{proof}
    By Proposition \ref{p: nonresonant nilpotence} and Corollary
    \ref{c: bound on diagonal grade} we have 
    \[\dia(\prescript{}{n}F_{n-1}({\alpha},;{\beta}\mid x)) , \had(\prescript{}{n}F_{n-1}({\alpha}; {\beta}\mid x)) \geq n.\]
    Then from \eqref{eq: hadamard product when betas are one} and the fact that
    $\prescript{}{1}F_0(\alpha_i; 1\mid x) = (1-x)^{-\alpha_i}$ 
    is algebraic we see
    $\had(\prescript{}{n}F_{n-1}({\alpha}; {\beta}\mid x)) = n$. The statement
    for diagonal grade follows from Lemma \ref{l: relate hg and dg}.
\end{proof}

\begin{corollary}
    \label{c: all had and dia grades exist}
    For every $k\geq 0$ there exists a function $f$ with
    $\dia(f)=k$ and $\had(f)=k$. In particular, $\mathcal{D}_k \subsetneq \mathcal{D}_{k+1}$ and $\mathcal{H}_k \subsetneq \mathcal{H}_{k+1}$.
\end{corollary}
\begin{remark}
    Rivoal and Roques \cite{rivoal2014hadamard} show that Theorem 
    \ref{t:all beta = 1} holds without the nonresonant condition, if you assume the Lang--Rohrlich conjecture.
\end{remark}

\paragraph{The general case} We now consider general 
hypergeometric functions, which includes the resonant case where
$\alpha_i - \beta_j \in \Z$ for some $i$ and $j$. It is known that the differential equation $L({\alpha};{\beta})$, and hence its monodromy representation, is reducible.

\begin{lemma}\label{l: factorization of hypergeometric odes}
	For any $\alpha,\beta$, the local system of solutions of $L(\alpha;\beta)$ is isomorphic to that of
	\begin{equation}\label{e:factored-ode}
		\left(\prod_{i=1}^a(\theta + \gamma_j)\right)L(\tilde{ \alpha};\tilde{ \beta})\left(\prod_{i=1}^b (\theta + \delta_i)\right).
		\end{equation}
		where $\tilde{\alpha},\tilde{\beta}$ are maximal subsets of $\alpha$ and $\beta$ respectively so that $(\tilde{\alpha},\tilde{\beta})$ is nonresonant. If $\alpha_i - \beta_j \notin \mathbb{N}$ for all $i,j$ then $a = 0$. If $\beta_j - \alpha_i \notin \mathbb{N}$ for all $i,j$ then $b=0$.
\end{lemma}
\begin{proof} This follows directly from the proof of \cite[Proposition 2.7]{beukers1989monodromy}. More precisely, if we let $V({ \alpha};{ \beta})$ denote the local system of solutions of $L({\alpha};{\beta})$, suppose without loss of generality, that $\beta_n - \alpha_n = m \in \mathbb{Z}$. If $m = -1$ then it is a direct observation that 
\[
L(\alpha_1,\dots, \alpha_n; \beta_1,\dots, \alpha_n+1) = L(\alpha_1,\dots,\alpha_{n-1};\beta_1,\dots, \beta_{n-1})(\theta + \beta_n). 
\]
and if $m = 0$, we see that 
\[
L(\alpha_1,\dots, \alpha_n; \beta_1,\dots, \alpha_n) = (\theta + \beta_{n-1}+1)L(\alpha_1,\dots,\alpha_{n-1};\beta_1,\dots, \beta_{n-1}). 
\]
If $m\neq -1$ then the proof of loc. cit. provides an isomorphism between the local systems $V({\alpha};{ \beta}) \cong V(\check{ \alpha};{ \beta})$ where 
\[
\check{ \alpha} = (\alpha_1,\dots, \alpha_n-m+1)
\]
if $m \leq -2$ and $V({ \alpha};{ \beta}) \cong V({ \alpha};\check{ \beta})$ where
\[
\check{ \beta} = (\beta_1,\dots, \beta_n+m-1)
\]
if $m \geq 0$. Repeating this procedure for every pair $\alpha_i, \beta_j$ so that $\alpha_i -  \beta_j \in \mathbb{Z}$, we find that $V({\alpha};{\beta})$ is isomorphic to the local system of solutions of an equation of the form \eqref{e:factored-ode}. 
\end{proof}

The factorization in \eqref{e:factored-ode} determines a filtration of local systems on $\mathbb{P}^1\smallsetminus \{0,1,\infty\}$:
\begin{equation}\label{eq: filtration coming from factorization}
0 \subseteq V_1 \subseteq V_2 \subseteq V_3 = V(\alpha;\beta), 
\end{equation}
where 
\begin{equation}\label{eq: gradeds of filtration}
    V_1\cong \mathrm{Sol}\left( \prod_{i=1}^b (\theta+\delta_i) \right), \quad V_2/V_1 \cong \mathrm{Sol}(L(\tilde{\alpha};\tilde{\beta})), \quad V_3/V_2 \cong \mathrm{Sol}\left( \prod_{j=1}^a (\theta + \gamma_j) \right).
\end{equation}
Let $L$ denote the monic operator of lowest degree in $\overline{\mathbb{Q}}(x)\left[\frac{d}{dx}\right]$ annihilating $F(\alpha;\beta\mid x)$. By the division algorithm for differential operators, we may write $L(\alpha;\beta) = L'L$ for some operator $L'$. Let $V$ denote $\mathrm{Sol}(L)$. Therefore $V$ is also a local subsystem of $L(\alpha;\beta)$.
\begin{lemma}\label{l: surjective map onto reduced weights}
    Suppose $\prescript{}{n}F_{n-1}(\alpha;\beta\mid x)$ is not rational. The local system $V(\tilde{\alpha};\tilde{\beta})$ is isomorphic to a subquotient of $V$.
\end{lemma}
\begin{proof}
  Suppose $V\cap V_2$ is not contained in $V_1$. Then $(V\cap V_2)/(V\cap V_1)$ is a nonzero subobject of $V_2/V_1\cong V(\tilde{\alpha};\tilde{\beta})$. Since $V(\tilde{\alpha};\tilde{\beta})$ is irreducible \cite[Proposition 3.3]{beukers1989monodromy} we must have that $(V\cap V_2)/(V \cap V_1) \cong V(\tilde{\alpha};\tilde{\beta})$. The remainder of the proof shows that $V\cap V_2$ is not contained in $V_1$.

    Let us first establish notation. A matrix $T$ is a pseudo-reflection if $\mathrm{rank}(T-\mathrm{Id}) = 1$. According to e.g. \cite[Proposition 2.8, Proposition 2.10]{beukers1989monodromy}, the monodromy of a hypergeometric local system at $1 \in \mathbb{P}^1$ is either a pseudo-reflection or the identity. It is a pseudo-reflection if the hypergeometric local system is nonresonant. Therefore, if $T_1$ denotes the monodromy operator of $V_3$ at 1, $(T_1-\mathrm{Id}): V_2/V_1 \rightarrow V_2/V_1$ has rank 1. Thus $(T_1-\mathrm{Id}) : V_2 \rightarrow V_2$ has rank at least 1 and its image is not contained in $V_1$. Since $(T_1-\mathrm{Id}) : V_3\rightarrow V_3$ has rank at most 1, it follows that $(T_1-\mathrm{Id})(V_3) = (T_1-\mathrm{Id})(V_2)$. In particular, $(T_1-\mathrm{Id})(V_3) \cap V_1 =\{0\}$.

    Now let $v\in V$ and assume $V\cap V_2 \subseteq V_1$. We show first that $V$ must have trivial monodromy over 1. Note that both $V_3/V_2$ and $V_1$ are local systems which have trivial monodromy over 1. Therefore, $(T_1-\mathrm{Id})v \in V \cap V_2\subseteq V_1$. As we argued in the previous paragraph, $V_1\cap (T_1-\mathrm{Id})(V) = \{0\}$. Thus $(T_1-\mathrm{Id})v =0.$

    Therefore, $L$ is a linear differential operator whose solutions have nontrivial monodromy only at 0 and $\infty$. Since $\prescript{}{n}F_{n-1}(\alpha;\beta\mid x)$ is annihilated by $L$, and $\prescript{}{n}F_{n-1}(\alpha;\beta\mid x)$ is holomorphic at 0, it must extend meromorphically to all of $\mathbb{C}$. Since $L$ is Fuchsian, $\prescript{}{n}F_{n-1}(\alpha;\beta\mid x)$ is also meromorphic at $\infty$ so it is rational.\end{proof}

Christol defines the height of a hypergeometric function to be: 
\[
h({{\alpha},{\beta}}) = \#\{ j \in \{1,\dots,n\} \mid \beta_j \in \mathbb{Z}\} - \#\{ i \in \{1,\dots, n\} \mid \alpha_i \in \mathbb{Z}\}.
\]
We take $\beta_n = 1$. If $(\tilde{\alpha},\tilde{\beta})$ are as in the statement of Lemma \ref{l: factorization of hypergeometric odes} $|h(\alpha,\beta)| = \max\{\#\{j \mid \tilde{\beta}_j \in \mathbb{Z}\},\# \{i \mid \tilde{\alpha}_i \in \mathbb{Z}\}\}$. 

\begin{theorem}
    Suppose $\alpha_i \notin \mathbb{Z}_{<0}$ for all $i$. Then the diagonal and Hadamard grades of $\prescript{}{n}F_{n-1}(\alpha;\beta\mid x)$ are bounded below by $|h(\alpha;\beta)|$.
\end{theorem}
\begin{proof}
    Suppose $\prescript{}{n}F_{n-1}(\alpha;\beta \mid x)$ is not rational. By Lemma \ref{l: surjective map onto reduced weights}, $V(\tilde{\alpha};\tilde{\beta})$ is a subquotient of $V$.
    From this we see that the nilpotent index of $L$ is greater
    than the nilpotent index of $L(\widetilde{\alpha},\widetilde{\beta})$ at $0$ and at $\infty$. Levelt's Theorem \cite{levelt1961hypergeometric} says the nilpotence
    index of $L(\widetilde{\alpha},\widetilde{\beta})$ at zero is at least $\#\{i \mid \tilde{\beta}_i \in \mathbb{Z}\}$ and the nilpotence
    index of $L(\widetilde{\alpha},\widetilde{\beta})$ at $\infty$ is at least $ \#\{i \mid \tilde{\alpha}_i \in \mathbb{Z}\}$. Applying Corollary \ref{c: bound on diagonal grade} gives the result.

        Now we address the case where $\prescript{}{n}F_{n-1}(\alpha;\beta\mid x)$ is rational. In this case $\mathbf{hg} = \mathbf{dg} =0$ so it remains to show that $h(\alpha;\beta) = 0$. Our proof is an adaptation of the proof of \cite[Theorem 2.3]{cattani2001rational}. Since $\prescript{}{n}F_{n-1}(\alpha;\beta\mid x)$ is rational and holomorphic at 0, and it is a solution to $L(\alpha;\beta)$ whose singularities appear at $0,1,\infty$, we must have $\prescript{}{n}F_{n-1}(\alpha;\beta\mid x) = P(x)/(1-x)^k$ for some $k$ and $P(x)$ a polynomial in $x$. Since $\alpha_i \notin \mathbb{Z}_{<0}$ for all $i$, we see that $\prescript{}{n}F_{n-1}(\alpha;\beta\mid x)$ is not a polynomial and thus $k\geq 1$. If we write $P(x)/(1-x)^k = \sum_{i=0}^\infty a_ix^i$, one can show that there is a polynomial $f(z)$ so that $f(i) = a_i$ for $i\gg 0$. In fact, since \[\frac{(-1)^{k-1}}{(k-1)!}\frac{d}{dx}^{k-1} \frac{1}{1-x} = \frac{1}{(1-x)^k}\] we have
        \[ \frac{1}{(1-x)^k} = (-1)^{k-1} \sum_{i\geq 0} b(i)x^i,\]
        where $(k-1)!b(x)=(x+1)\cdots (x+k-1)$. Then if $P(x) = \sum_{i=0}^d c_ix^i$ we find that for $i> d-k$ we have $a_i=f(i)$ with
        \[ f(x) = \sum_{i=0}^d c_ib(x-i).\]
        In other words,
    \[f(i) = \dfrac{(\alpha_1)_i\dots (\alpha_n)_i}{(\beta_1)_i\dots (\beta_{n-1})_i (1)_i},\qquad i \gg 0.  \]
    We let
    \[ \gamma(z) = \dfrac{f(z+1)}{f(z)},\qquad \eta(z) = \dfrac{\prod_{i=1}^n  (\alpha_i + z)}{\prod_{i=1}^n (\beta_i + z)}. \]
    Since $\gamma(z)$ and $\eta(z)$ are rational functions which agree for infinitely many complex numbers, they must be equal. It is straightforward to check \cite[Lemma 2.1]{vanderput-Singer-book-difference} that for a rational function of the form $h(z):=g(z+1)/g(z), g(z) \in \mathbb{C}(z)$, 
    \[\sum_{\alpha \in z_0+\mathbb{Z}} \mathrm{ord}_\alpha(h) = 0\]
    for any $z_0 \in \mathbb{C}.$ Therefore, since the poles of $\gamma(z)$ appear at $\{-\beta_j \mid j=1,\dots n\}$ and the zeroes appear at $\{-\alpha_i \mid i=1,\dots, n\}$, for each $z_0 \in \mathbb{C}$,
    \[
    \#\{j\in  \{1,\dots,n\}\mid \beta_j \equiv z_0  \bmod \mathbb{Z}\} = \#\{i\in  \{1,\dots,n\}\mid \alpha_i \equiv z_0  \bmod \mathbb{Z}\}. 
    \]
    If $z_0 = 0$, this is just the statement that $h(\alpha;\beta) = 0$.
\end{proof}

\subsection{Other examples}
We can use Corollary \ref{c: bound on diagonal grade} to determine
the diagonal grade of other $D$-finite functions, provided we
have a diagonal representation with few enough variables and we understand
the local monodromy around $0$.

\begin{example}\label{e: Apery}
    Recall the Apéry numbers,
    \[
    A(n) = \sum_{k=0}^n{n \choose k}^2 {n+k \choose k}^2
    \]
    and consider the generating function
    \[
    G(x) = \sum_{n=0}^\infty A(n) x^n.
    \]
    There are several ways to represent $G(x)$ as a diagonal, e.g. \cite{Christol_diagonals},
    that have representations of $G(x)$ as a diagonal in five variables. 
    This was improved by Straub \cite[Theorem 1.1]{Straub-Apery} who
    discovered that
    \[
    G(x) = \Delta_3\left(\dfrac{1}{(1-x_0-x_1)(1-x_3-x_2) + x_0x_1x_2x_3}\right).
    \]
    The recurrence for the Apéry numbers, discovered by Apéry himself, shows that $G(x)$ is a solution to the regular singular differential equation
    \[
    (x^4 - 34x^3 + x^2)\left(\dfrac{d}{dx}\right)^3 + (6x^3 - 153x^2 + 3x) \left(\dfrac{d}{dx}\right)^2 + (7x^2 - 112x + 1)\left(\dfrac{d}{dx}\right) + (x-5).
    \]
    The local exponents of this operator are all 0 at the origin, and
    we compute that $T_0$ is conjugate to a  3 dimensional Jordan block (see \cite[Corollary 2]{Beukers-Peters-Zeta3}.) Therefore, Corollary \ref{c: bound on diagonal grade} says $G(x)$ 
    has diagonal grade at most 3. Since $G(x)$ is the diagonal of a function in 4 variables, we see $\dia(G(x))=3$. In other words,
    Straub's representation of $G(x)$ cannot be improved!
\end{example}
\begin{example}\label{e: banana famiily}
    We present an infinite family of non-hypergeometric diagonals whose diagonal grade we expect to be able to compute. Define
    \[
    B_\ell(n) = \sum_{r_0+ \dots + r_\ell = n}{ n \choose r_0,\dots , r_\ell}^2,\qquad H_\ell(x) = \sum_{n=0}^\infty B_\ell(n) x^n. 
    \]
    It is not hard to check that $B_\ell(n)$ is the constant term of $g^n_\ell$ where
    \[
    g_\ell = (x_1 +\dots + x_\ell+1)\left(\dfrac{1}{x_1} + \dots + \dfrac{1}{x_\ell} +1\right)
    \]
    Suppose $f$ is a Laurent polynomial in variables $x_1,\dots, x_{m}$ so that $x_1\dots x_m f$ is a polynomial. Let $\mathrm{ct}(f^n)$ denote the constant term of $f^n$, and $\mathrm{cts}(f) = \sum_{n=0}^\infty \mathrm{ct}(f^n)x^n$. An easy geometric series argument shows that the diagonal of $1/(1-(x_1\dots x_{m+1})f)$ is $\mathrm{cts}(f)$. Therefore, since $H_\ell(x) = \mathrm{cts}(g_\ell)$ is the diagonal of a function in $\ell+1$ variables. On the other hand, it is argued in \cite[\S\S 3.1, 3.2, 3.3]{bonisch2021analytic} that $H_\ell(x)$ is a solution to a differential operator $L_\ell$ of rank $\ell$ and which has maximally unipotent monodromy at 0. In \cite[\S 4]{lairez2023algorithms}, Lairez and Vanhove compute the same differential operators for $\ell \leq 5$ and check via computer algebra calculation that they are irreducible. Therefore, $\mathbf{dg}(H_\ell(x)) = \ell$ if $\ell \leq 5$. We expect that the same statement holds for $\ell \geq 6$.
\end{example}
\begin{remark}
     The functions $H_\ell(x)$ in Example \ref{e: banana famiily}, and their generalizations, are maximal cut Feynman integrals of equal-mass banana graphs in quantum field theory. They also appear in mirror symmetry as the quantum period of certain families of Calabi--Yau varieties. See \cite{bonisch2021analytic} and the references therein.
\end{remark}

Similar examples are abundant in the literature. We point the reader to
\cite{hassani2025diagonals_conjecture} for an extensive collection of examples.

\section{The ring of diagonals} \label{s: the ring of diagonals}
In this section we establish some basic properties
about the ring $(\mathcal{D},+,*)$, where the multiplicative
operation is given by the Hadamard product.

\subsection{$\mathcal{D}$ is a filtered ring}

\begin{proposition} \label{p:filtered}
    For $k_1,k_2\geq 0$ we have $\mathcal{D}_{k_1}*\mathcal{D}_{k_2} \subseteq \mathcal{D}_{k_1 + k_2}.$
\end{proposition}
\begin{proof}
    This is an application of Lemma \ref{l: relate hg and dg}.
    Let $f_1 \in \mathcal{D}_{k_1}$ and $f_2\in \mathcal{D}_{k_2}$.
    By Lemma \ref{l: relate hg and dg} there are functions
    $g_1 \in K\cbrac{x_0,\dots,x_{k_1-1}}$ and $g_2 \in K\cbrac{x_0,\dots,x_{k_2-1}}$ with $f_1=\Delta_{k_1-1}(g_1)$ and 
    $f_2=\Delta_{k_2-1}(g_2)$. As
    $g_1(x_0,\dots, x_{k_1-1})\cdot g_2(x_{k_1}, \dots, x_{k_1+k_2-1})$
    is in $K\cbrac{x_0,\dots,x_{k_1+k_2-1}}$, we know
    from Lemma \ref{l: relate hg and dg} that there exists $h \in
    K[x_0,\dots,x_{k_1+k_2}]_{(x_0,\dots,x_{k_1+k_2})}$ with
    \[ \Delta_{k_1+k_2-1}(g_1(x_0,\dots, x_{k_1-1})\cdot g_2(x_{k_1}, \dots, x_{k_1+k_2-1})) = \Delta_{k_1+k_2}(h).\]
    We then compute
    \begin{align*}
        f_1 * f_2& = \Delta_{k_1-1}(g_1) * \Delta_{k_2-1}(g_2) \\
        &=  \Delta_{k_1+k_2-1}(g_1(x_0,\dots, x_{k_1-1})\cdot g_2(x_{k_1}, \dots, x_{k_1+k_2-1})) \\
        &= \Delta_{k_1+k_2}(h).\qedhere
    \end{align*}
\end{proof}

\begin{remark}
    We know that $\mathcal{H}_1=\mathcal{D}_1$. Thus, if
    the inclusion in Proposition \ref{p:filtered} were
    an equality, we would have $\mathcal{H}_k=\mathcal{D}_k$ for
    all $k\geq 0$. Rivoal and Roques \cite{rivoal2014hadamard} show that, under the assumption of the Lang--Rohrlich conjecture, the hypergeometric function  $\prescript{}{3}F_{2}(\frac{1}{7},\frac{2}{7}, \frac{4}{7};\frac{1}{2},1\mid x)$ 
    has infinite Hadamard grade. However, $\prescript{}{3}F_{2}(\frac{1}{7},\frac{2}{7}, \frac{4}{7};\frac{1}{2},1\mid x)$ is globally bounded
    by work of Christol \cite{Christol_bounded_hypergeometric}.
    Thus, if one believes Christol's conjecture that globally bounded
    functions are diagonals and the Lang--Rohrlich conjecture, we see that being a diagonal
    does not imply finite Hadamard grade. In particular, the inclusion
    in Proposition \ref{p:filtered} is not an equality. However,
    it is worth pointing out that Christol suggested 
    $\prescript{}{3}F_{2}(\frac{1}{7},\frac{2}{7}, \frac{4}{7};\frac{1}{2},1\mid x)$ as a potential counterexample to his conjecture,
    and thus far it is unknown if it is a diagonal.  
\end{remark}

\subsection{The action of $\mu_m$ on $\mathcal{D}$ and zero divisors}
For any $m\geq 1$ and $r$ satisfying $0\leq r \leq m-1$ we define
a $K$-linear operator as follows:
\begin{align*}
    e_{r,m}:K\dbrac{x} &\to K\dbrac{x},\\
    \sum_{k=0}^\infty a_k x^k & \mapsto \sum_{k=0}^\infty a_{r+km}x^k.
\end{align*}
Let $m_0 \geq 1$ be coprime to $m$ and let $s$ satisfy $0\leq s \leq m_0-1$ we have
\begin{equation*}
    e_{s,m_0}\circ e_{r,m} = e_{r,m}\circ e_{s,m_0} = e_{t,m_0m},
\end{equation*}
where $t$ is the unique integer $0\leq t \leq mmm_0-1$ that is congruent
to $s$ modulo $m_0$ and $r$ modulo $m$. Furthermore, we have a decomposition:
\begin{align}\label{e: eigen decomposition}
    f(x) &= \sum_{r=0}^{m-1} x^r e_{r,m}(f)(x^m). 
\end{align}

\begin{remark}
    Let $f(x) \in K\dbrac{x}$ be a $D$-finite series and let $N_f$
    be the differential module as defined in Definition \ref{d: nilpotence exponent}. Let $[m]:\Spec(K(x)) \to \Spec(K(x))$ be the map
    corresponding to $x \mapsto x^m$. Then $N_{e_{r,m}(f)}$ is a direct factor of $[m]_*N_f$, the pushforward of $N_f$ along $[m]$. In particular, we see that $e_{r,m}(f)$ is again $D$-finite. 
\end{remark}

\begin{proposition}
    \label{p: e preserves diagonal}
    Let $f \in \mathcal{D}_k$. Then $e_{r,m}(f) \in \mathcal{D}_k$.
    Similarly, if $f \in \mathcal{H}_k$ then $e_{r,m}(f) \in \mathcal{H}_k$. 
\end{proposition}

\begin{proof}
    Let $h \in K(x_0,\dots,x_n)$ be a rational function with $f=\Delta_n(h)$. Let $h_0=\frac{h}{(x_0\dots x_n)^r}$. There
    exists $g \in K(x_0,\dots,x_n)$ such that
    \begin{equation*}
        g(x_0^{m}, \dots, x_n^{m}) = \frac{1}{m(n+1)}\sum_{i=0}^n \sum_{\zeta^m=1} h_0(x_0, \dots, x_{i-1}, \zeta x_i, x_{i+1}, \dots, x_n).
    \end{equation*}
    By construction we see that $e_{r,m}(f)=\Delta_n(g)$.
\end{proof}

By combining Proposition \ref{p: e preserves diagonal} with equation 
\eqref{e: eigen decomposition} we see that $\mathcal{D}$ has many zero divisors. Indeed,
for $0 \leq r,s \leq m-1$ with $r \neq s$ we have
\begin{align*}
    [x^r e_{r,m}(f)(x^m)] * [x^se_{s,m}(g)(x^m)] &= 0. 
\end{align*}
It turns out that this explains all the zero divisors in $\mathcal{D}$. 
\begin{theorem}
    \label{t: zero divisors}
    Let $f,g \in \mathcal{D}$. The following are equivalent:
    \begin{enumerate}[label=(\Roman*)]
        \item $f * g = 0$.
        \item There exists $n$ and subsets $F,G \subset \{0,\dots,m-1\}$ such that $F \cup G = \{0,\dots,m-1\}$
        and for $r \in F$ (resp. $s \in G$) we have
        $e_{r,m}(f)=0$ (resp. $e_{s,m}(f) = 0$). 
    \end{enumerate}
\end{theorem}

\begin{proof}
    It is immediate that $(II)$ implies $(I)$. For the converse,
    we use a weak generalization of the Skolem--Mahler--Lech theorem
    for certain $D$-finite series proven by Bézivin \cite{Bezivin-SML_theorem}. 
    This result says that for a $D$-finite series $h$ whose underlying
    differential equation has regular singularities at $0$ and $\infty$, there exists $m$ such that either i) $e_{r,m}(h)=0$ or ii) The coefficients of $e_{r,m}(h)$ that are zero have density zero. As any element of $\mathcal{D}$ is the solution to
    a Picard--Fuchs equation, we know that all singularities
    are regular \cite{katz_nilpotent_monodromy_theorem}, and thus we may apply Bézivin's result to
    $f$ and $g$. Choose $n$ so that Bézivin's theorem applies to both $f$ and $g$. Increasing $m$ if necessary, let $F$ be the subset of
    $\{0,\dots,m-1\}$ consisting of elements $r$ with $e_{r,m}(f)=0$.
    Define $G$ analogously. Assume there is
    an element $s \in \{0,\dots,m-1\}$ not contained in $F$ or $G$.
    Then the coefficients of $e_{s,m}(f)$ and $e_{s,m}(g)$ that are
    zero have density zero. Thus, the coefficients of $e_{s,m}(f * g)$
    that are zero have density zero. Thus, $e_{s,m}(f*g)\neq 0$, which is a contradiction. It follows that $F \cup G=\{0,\dots,m-1\}$.    
\end{proof}

\begin{corollary}
    \label{c: zero divisors} A diagonal $f \in \mathcal{D}$
    is a zero divisor if and only if there exists $r$ and $m$
    with $e_{r,m}(f)=0$.
\end{corollary}

    The operators $e_{r,m}$ can be similarly used to construct 
    additional examples where $\dia(f*g) < \dia(f) + \dia(g)$ and
    $\had(f*g) < \had(f) + \had(g)$. For example,
    set
    \begin{align*}
        f(x) &= e_{0,2}\left (\frac{1}{1-x}\right)(x^2) + x e_{1,2}\left (\sqrt{1-x}\right)(x^2) \\
        g(x) &= e_{0,2}\left (\sqrt{1-x}\right))(x^2) +  xe_{1,2}\left (\frac{1}{1-x}\right)(x^2).
    \end{align*}
    Then $\dia(f)=\dia(g)=1$, but we have $f*g=\sqrt{1-x}$, which
    has diagonal grade $1$. It is natural to ask if
    this is the only `obstruction' to $\dia(f*g)$ being
    equal to $\dia(f) + \dia(g)$.
    \begin{question}
        Assume that for every $m$ there exists
        $r$ such that $\dia(e_{r,m}(f))=\dia(f)$ and
        $\dia(e_{r,m}(g))= \dia(g)$. Is it true that
        $\dia(f * g) = \dia(f) + \dia(g)$? 
        We remark
        that this assumption is clearly necessary.
    \end{question}
    It is also natural to ask if we really need to check \emph{every}
    $m$ in the above question. This leads naturally to the
    following definition and question.
    \begin{definition}
        Let $f \in \mathcal{D}$. We say that $f$ has clean
        diagonal grade $k$ if $\dia(e_{r,m}(f))=k$ 
        for every $m$ and $r$. 
    \end{definition}
    For example, one can prove that $\prescript{}{n}F_{n-1}(\frac{1}{2},\dots, \frac{1}{2};1\dots,1 \mid x)$ has clean diagonal grade
    $n$. We can ask if a diagonal $f$ always decomposes into `clean' parts.
    \begin{question}
        Let $f \in \mathcal{D}$. Does there exist $m$ such that
        $e_{r,m}(f)$ has clean diagonal grade $k_r$ for each
        $0 \leq r \leq m-1$?
    \end{question}

\bibliography{hadamard}{}
\bibliographystyle{alpha}

\end{document}